\documentclass[a4paper,12pt,twoside,leqno,final]{amsart}
\usepackage{amsmath}
\usepackage{amssymb}

\newtheorem{thm}{Theorem}[section]
\newtheorem{lem}[thm]{Lemma}

\newtheorem{cor}[thm]{Corollary}

\newtheorem*{tha}{Theorem A}
\newtheorem*{thb}{Theorem B}
\newtheorem*{thc}{Theorem C}

\newcommand{\C}{{\mathbb C}}
\newcommand{\D}{{\mathbb D}}
\newcommand{\R}{{\mathbb R}}
\newcommand{\T}{{\mathbb T}}
\newcommand{\Z}{{\mathbb Z}}
\newcommand{\N}{{\mathbb N}}

\newcommand{\J}{{\mathcal J}}

\newcommand{\bmo}{{\rm BMO}}
\newcommand{\bmoa}{{\rm BMOA}}

\newcommand{\Al}{A^\alpha}

\newcommand{\HN}{H^\infty_N}
\newcommand{\hnone}{H^\infty_{n+1}}

\newcommand{\M}{{\mathcal{M}}}
\newcommand{\La}{\Lambda}

\renewcommand{\sb}{\subset}
\newcommand{\dist}{\operatorname{dist}}
\newcommand{\eps}{\varepsilon}

\newcommand{\f}{\frac}
\newcommand{\ov}{\overline}
\newcommand{\al}{\alpha}
\newcommand{\be}{\beta}

\newcommand{\de}{\delta}
\newcommand{\la}{\lambda}
\newcommand{\ze}{\zeta}
\renewcommand{\th}{\theta}

\newcommand{\ph}{\varphi}
\newcommand{\om}{\omega}
\newcommand{\Om}{\Omega}
\newcommand{\Omte}{\Om(\th,\eps)}

\newcommand{\const}{\text{\rm const}}

\newcommand{\specth}{\text{\rm spec}\,\th}

\numberwithin{equation}{section}

\input cyracc.def
\font\tencyr=wncyr10
\def\cyr{\tencyr\cyracc}

\title[Blaschke products and nonideal ideals]
{Blaschke products and nonideal ideals\\ 
in higher order Lipschitz algebras}
\dedicatory{To Victor Petrovich Havin, with admiration (best phrased as a palindrome):\\
{\cyr vot pedagog adeptov!}} 
\author{Konstantin M. Dyakonov}
\address{ICREA and Universitat de Barcelona, Departament de Matem\`atica 
Aplicada i An\`alisi, Gran Via 585, E-08007 Barcelona, Spain}
\email{dyakonov@mat.ub.es}
\thanks{Supported in part by grant MTM2008-05561-C02-01 from El Ministerio de Ciencia 
e Innovaci\'on (Spain) and grant 2009-SGR-1303 from AGAUR (Generalitat de Catalunya).}

\begin{document}
\begin{abstract}
We investigate certain ideals (associated with Blaschke products) of the analytic Lipschitz 
algebra $\Al$, with $\al>1$, that fail to be \lq\lq ideal spaces". The latter means that 
the ideals in question are not describable by any size condition on the function's modulus. 
In the case where $\al=n$ is an integer, we study this phenomenon for the algebra 
$H^\infty_n=\{f:f^{(n)}\in H^\infty\}$ rather than for its more manageable Zygmund-type 
counterpart. This is based on a new theorem concerning the canonical factorization 
in $H^\infty_n$. 
\end{abstract}

\maketitle

\section{Introduction and results}

Let $H^\infty$ stand for the algebra of bounded analytic functions on the disk 
$\D:=\{z\in\C:|z|<1\}$. Recall that a function $\th\in H^\infty$ is said to be {\it inner} 
if $|\th(\ze)|=\lim_{r\to1^-}|\th(r\ze)|=1$ at almost all points $\ze$ of $\T:=\partial\D$. 
Given $f\in H^\infty$ and an inner function $\th$, we trivially have $f\th^k\in H^\infty$ 
for every $k$ in $\N:=\{1,2,\dots\}$. Our plan is to show how little of this remains 
true -- and to discuss the subtleties that arise -- when $H^\infty$ gets 
replaced by a smaller algebra whose members are suitably smooth up to $\T$, 
especially when the order of smoothness exceeds 1. Specifically, we shall be concerned 
with the analytic {\it Lipschitz spaces}, to be defined in a moment. 
\par For $\al>0$, we write $\Al$ for the set of those $f\in H^\infty$ which satisfy 
\begin{equation}\label{eqn:mthder}
|f^{(m)}(z)|=O\left((1-|z|)^{\al-m}\right),\qquad z\in\D, 
\end{equation}
for some (any) integer $m$ with $m>\al$; here $f^{(m)}$ is the $m$th order derivative 
of $f$. 
\par It is well known that \eqref{eqn:mthder} does not actually depend on the choice of $m$, 
as long as $m>\al$, except for the constant in the $O$-condition. Also, for $0<\al<1$, a 
classical theorem of Hardy and Littlewood tells us that $\Al$ is formed by precisely those 
analytic functions $f$ on $\D$ which obey the Lipschitz condition of order $\al$, i.\,e., 
\begin{equation}\label{eqn:lipond}
|f(z)-f(w)|\le C|z-w|^\al,\qquad z,w\in\D,
\end{equation}
with some fixed $C=C_f>0$. Similarly, in the case $\al\in(0,\infty)\setminus\N$, 
an analytic function $f$ will be in $\Al$ if and only if $f^{(n)}$ 
satisfies the Lipschitz condition of order $\al-n$, 
where $n=[\al]$ is the integral part of $\al$. The space $A^1$ is known as the analytic 
{\it Zygmund class}; the higher order Zygmund classes $A^n$ (i.\,e., the $\Al$-spaces 
with $\al=n\in\{2,3,\dots\}$) are related to it by the formula $A^n=\{f:f^{(n-1)}\in A^1\}$. 

\par Furthermore, we shall be dealing with the algebras 
$$H^\infty_n:=\{f\in H^\infty:\,f^{(n)}\in H^\infty\},\qquad n\in\N.$$ 
Of course, $H^\infty_1$ coincides with the set of those $f\in H^\infty$ which satisfy the 
Lipschitz condition \eqref{eqn:lipond} with $\al=1$. Let us also recall that $H^\infty_n$ is 
properly contained in $A^n$, for each $n$. 

\par Besides, we need the \lq real variable' Lipschitz--Zygmund spaces $\La^\al=\La^\al(\T)$, 
which actually consist of complex-valued functions and can be defined by $\La^\al:=\Al+\ov{\Al}$. 
An equivalent, and more traditional, definition is as follows: for $0<\al<\infty$, the space 
$\La^\al$ consists of the functions $f\in C(\T)$ such that 
$$\|\Delta_h^mf\|_\infty=O(|h|^\al),\qquad h\in\R,$$ 
where $\|\cdot\|_\infty$ is the $\sup$-norm on $\T$, $m$ is an integer with $m>\al$, 
and $\Delta_h^m$ stands for the $m$th order difference operator with step $h$. (As usual, the difference 
operators $\Delta_h^k$ are defined by induction: one puts $(\Delta_h^1f)(\ze):=f(e^{ih}\ze)-f(\ze)$ and 
$\Delta_h^kf=\Delta_h^1\Delta_h^{k-1}f$.) Finally, let us observe that $\La^\al$ is an algebra, for 
each $\al>0$, and so is $\Al=\La^\al\cap H^\infty$. 

\par Now let $\al\in(0,\infty)$ and fix a function $h\in H^\infty$. This done, consider the set 
\begin{equation}\label{eqn:defialh}
\mathcal I=\mathcal I(\al,h):=\{f\in\Al:\,fh\in\Al\}.
\end{equation}
Clearly, $\mathcal I$ is an {\it ideal} of the algebra $\Al$. Indeed, it is a linear (possibly 
nonclosed) subspace thereof, and one has $fg\in\mathcal I$ whenever $f\in\mathcal I$ 
and $g\in\Al$. Our aim is to study the ideals $\mathcal I(\al,h)$ that arise when $h=\th^k$, 
with $\th$ inner and $k\in\N$. Later on, we shall also look at similar ideals in $H^\infty_n$, 
but let us stick to the $\Al$ case for the time being. 

\par Two questions will be addressed. The first of these concerns the relationship 
between $\mathcal I(\al,\th^k)$ and $\mathcal I(\al,\th^l)$ 
for $k\ne l$. Secondly, we ask whether the functions $f$ from $\mathcal I(\al,\th^k)$ can 
be nicely described in terms of their moduli. After all, since a nontrivial inner function 
$\th$ is highly discontinuous at some points of $\T$, one feels that the inclusion $f\th^k\in\Al$ 
can only hold if $|f|$ becomes appropriately small near the singular set of $\th$. Precisely speaking, we 
want to know if/when $\mathcal I(\al,\th^k)$ is an \lq\lq ideal space", or rather an {\it ideal subspace} 
of $\Al$, a (fairly standard) notion that we shall now recall. Note, however, the new meaning attached 
to the word \lq\lq ideal". 

\par Suppose $X$ is a subspace, not necessarily closed, of a function space $Y$. 
We say that $X$ is an {\it ideal subspace} of $Y$ if, given any $f\in X$ and 
any $g\in Y$ with $|g|\le|f|$, it follows that $g\in X$. Roughly speaking, 
this means that the elements of $X$ are describable, among all functions in $Y$, by a certain 
\lq\lq size condition" on the function's modulus. Of course, the inequality $|g|\le|f|$ in the above 
definition is supposed to hold everywhere -- or perhaps almost everywhere -- on the underlying set. 
Throughout this paper, the bigger space $Y$ is taken to be either $\Al$ or $H^\infty_n$, so the set 
in question is $\D$. 

\par There seems to be no chance of confusion between the adjective \lq\lq ideal", as used 
in the preceding paragraph, and the noun \lq\lq ideal" that appears elsewhere, e.\,g., in the 
sentence following \eqref{eqn:defialh}. Moreover, we shall repeatedly refer to {\it ideal ideals} 
in $\Al$ or in $H^\infty_n$ (these are, by definition, ideals of the corresponding algebra that are also 
ideal subspaces thereof) and to {\it nonideal ideals} (i.\,e., the ones that fail to be ideal subspaces). 
In particular, \lq\lq nonideal" is always an adjective. 

\par The following theorem, to be found in \cite{DSpb, DAmer}, provides a criterion for a function 
$f\in\Al$ to be multipliable or divisible by (a power of) an inner function $\th$. See also 
\cite{DActa, DAdv, Dyn2} for alternative versions and approaches. The criterion will be stated 
in terms of a decrease condition to be satisfied by $f$ along the set 
$$\Omte:=\{z\in\D:|\th(z)|<\eps\},\qquad0<\eps<1.$$ 

\begin{tha} Let $0<\al<\infty$ and let $m$ be an integer with $m>\al$. Given $f\in\Al$ and 
an inner function $\th$, the following conditions are equivalent. 

\smallskip{\rm (i.A)} $f\th^m\in\Al$. 

\smallskip{\rm (ii.A)} $f\th^k\in\La^\al$ for all $k\in\Z$. 

\smallskip{\rm (iii.A)} $f/\th^m\in\La^\al$. 

\smallskip{\rm (iv.A)} For some $\eps\in(0,1)$, one has 
$$|f(z)|=O\left((1-|z|)^\al\right),\qquad z\in\Omte.$$ 
\end{tha}

Yet another equivalent condition is obtained from (iv.A) upon replacing the word \lq\lq some" 
by \lq\lq each"; see \cite{DSpb} or \cite{DAmer}. 

\par In what follows, we shall restrict our attention to the $\Al$-spaces with $\al\notin\N$. The 
Zygmund classes $A^n$ will not reappear (except, very briefly, in the proof of Theorem \ref{thm:factint}). 
For integral values of the smoothness exponent, we are going to consider the $H^\infty_n$ spaces instead; 
we shall be back to this in a while. 

\par We know from Theorem A that, for $k\in\N\cap(\al,\infty)$, the 
set $\mathcal I(\al,\th^k)$ does not depend on the choice 
of $k$ and is an ideal ideal of $\Al$. Indeed, if $f$ and $g$ are $\Al$-functions with 
$|g|\le|f|$ on $\D$, then $g$ is sure to satisfy (iv.A) whenever $f$ does. Now what about 
$\mathcal I(\al,\th^k)$ in the range $k\in\N\cap(0,\al)$? Of course, the question is meaningful 
for $\al>1$ only, and this time a different picture comes into sight. Namely, we shall 
soon see that the ideals $\mathcal I(\al,\th^k)$ with $1\le k\le[\al]$ may happen to be pairwise 
distinct, all of them strictly larger than $\mathcal I(\al,\th^{[\al]+1})$, and also nonideal. 

\par Meanwhile, we observe that one always has 
$$\mathcal I(\al,\th^k)\sb\mathcal I(\al,\th^l)\quad\text{for}\quad k\ge l\ge1.$$ 
This is due to the so-called {\it $f$-property} (or {\it division property}) of $\Al$ that 
was established by Havin in \cite{H}: whenever $F\in H^\infty$ and $I$ is an inner function with 
$FI\in\Al$, we actually have $F\in\Al$. We mention in passing that the $f$-property, and in 
fact a certain stronger property involving coanalytic Toeplitz operators, was verified in \cite{H} 
for a large number of important \lq\lq smooth analytic classes". This laid the foundation for 
much of subsequent research in the field; see \cite{Shi2} for an overview and further developments. 

\par Now let us recall that every inner function $\th$ can be factored canonically as 
$\th=\la BS$, where $\la$ is a unimodular constant, $B$ is a {\it Blaschke product}, and $S$ is a 
{\it singular inner function}; see \cite[Chapter II]{G}. More explicitly, the factors involved are 
of the form 
$$B(z)=B_{\{z_j\}}(z):=\prod_j\f{\bar z_j}{|z_j|}\f{z_j-z}{1-\bar z_jz},$$ 
where $\{z_j\}\sb\D$ is a sequence (possibly finite or empty) with $\sum_j(1-|z_j|)<\infty$, and 
$$S(z)=S_\mu(z):=\exp\left\{-\int_\T\f{\ze+z}{\ze-z}
d\mu(\ze)\right\},$$
where $\mu$ is a (nonnegative) singular measure on $\T$. The closure of the set 
$\{z_j\}\cup\text{\rm supp}\,\mu$ is called the {\it spectrum of $\th$}; we shall denote it 
by $\specth$. 

\par If $\th=S$ is a singular inner function, then we still have 
$$\mathcal I(\al,\th)=\mathcal I(\al,\th^2)=\dots,$$ 
for each $\al\in(0,\infty)$, just as it happens for an arbitrary inner function in the range 
$\al\in(0,1)$. Moreover, the ideal $\mathcal I(\al,\th)$ is then ideal, as a subspace of $\Al$. 
To verify these claims, write $\th=S_0^m$, where $S_0$ is another singular inner function 
and $m$ is an integer with $m>\al$; then apply Theorem A to $S_0$. 

\par However, Blaschke products -- and hence generic inner functions -- may exhibit a different 
type of behavior. This was discovered by Shirokov (see \cite{Shi1} or \cite[Chapter I]{Shi2}) who 
came up with an ingenious construction of a function $g\in A^\infty:=\bigcap_{0<\al<\infty}\Al$ 
and a Blaschke product $B$ such that $g/B\in A^\infty$, but $gB\notin\bigcup_{\al>1}\Al$. 
It follows then, for any fixed $\al>1$, that the function $f:=g/B$ lies in 
$\mathcal I(\al,B)\setminus\mathcal I(\al,B^2)$. In addition, the ideal $\mathcal I(\al,B)$ is now 
nonideal: indeed, we have $f\in\mathcal I(\al,B)$ and $g\in\Al\setminus\mathcal I(\al,B)$, 
whereas $|g|\le|f|$ on $\D$. 

\par Our current purpose is a more detailed analysis of the (possibly) nonideal ideals of the form 
$\mathcal I(\al,B^k)$, with $1\le k\le n$. Here and below, $\al$ takes values in $(1,\infty)\setminus\N$ 
and $n=[\al]$ is the integral part of $\al$. Furthermore, a special class of Blaschke products will be 
singled out and dealt with. Namely, it will be assumed that $B=B_{\{z_j\}}$ is an {\it $\Al$-interpolating 
Blaschke product}, in the sense that its zeros $z_j$ are all simple and their closure $E:={\rm clos}\{z_j\}$ 
is an {\it $\Al$-interpolating set}. The latter means, in turn, that one can freely prescribe the values 
of an $\Al$-function and its successive derivatives (of order at most $n$) on $E$, within a certain 
natural class of data. The formal definition and a geometric description of such sets will be recalled 
in Section 2 below. In particular, it turns out that the class of $\Al$-interpolating sets (and hence 
Blaschke products) does not depend on $\al$. The sets in question are thus the same as $A^\be$-interpolating 
sets for some (any) $\be\in(0,1)$, and these are easy to define: we call $E$ an $A^\be$-interpolating set 
if every function $\ph:E\to\C$ satisfying 
\begin{equation}\label{eqn:lipone}
|\ph(z)-\ph(w)|\le\const\cdot|z-w|^\be,\qquad z,w\in E,
\end{equation}
can be written as $\ph=f|_E$ for some $f\in A^\be$. 

\par Once $\al$ and the Blaschke product $B$ are fixed, we write $\J_k:=\mathcal I(\al,B^k)$, so that 
$$\J_k=\{f\in\Al:\,fB^k\in\Al\},\qquad k\in\N.$$ 
In view of the above discussion, we always have 
\begin{equation}\label{eqn:chain}
\J_1\supset\dots\supset\J_n\supset\J_{n+1} 
\end{equation}
and 
\begin{equation}\label{eqn:tail}
\J_{n+1}=\J_{n+2}=\dots. 
\end{equation}
The ideals in \eqref{eqn:tail} are ideal subspaces of $\Al$, as we know, but the preceding 
ones (i.\,e., $\J_1,\dots,\J_n$) may well be nonideal. Our first 
result provides a criterion for this to happen, and also for 
the inclusions in \eqref{eqn:chain} to be proper (strict). When stating it, and later on, we shall 
use the notation $d_j$ for the quantity ${\rm dist}\left(z_j,\{z_l\}_{l\ne j}\right)$ associated 
with the zero sequence $\{z_j\}$ of $B$. In other words, 
$$d_j:=\inf\left\{|z_j-z_l|:\,l\in\N\setminus\{j\}\right\}.$$ 

\begin{thm}\label{thm:pervaya} Let $n\in\N$, $n<\al<n+1$, and suppose $B$ is an $\Al$-interpolating 
Blaschke product with zeros $\{z_j\}$. The following are equivalent. 

\smallskip{\rm (i.1)} The ideals $\J_1,\dots,\J_n$ are all nonideal (as subspaces of $\Al$). 

\smallskip{\rm (ii.1)} The inclusions $\J_1\supset\dots\supset\J_n\supset\J_{n+1}$ are all proper. 

\smallskip{\rm (iii.1)} $\J_n$ is nonideal. 

\smallskip{\rm (iv.1)} $\J_{n+1}\ne\J_n$. 

\smallskip{\rm (v.1)} One has 
\begin{equation}\label{eqn:crucon}
\sup_j\f{d_j}{1-|z_j|}=\infty. 
\end{equation}
\end{thm}

\par To get an example of an $\Al$-interpolating Blaschke product $B=B_{\{z_j\}}$ that satisfies 
\eqref{eqn:crucon}, take $z_j=(1-a^j)\exp(ib^j)$, $j\in\N$, where $a$ and $b$ are any fixed numbers 
with $0<a<b<1$. On the other hand, the Blaschke product with zeros $z_j=1-a^j$, where $a\in(0,1)$, 
is $\Al$-interpolating and violates \eqref{eqn:crucon}. 

\par A restricted version of Theorem \ref{thm:pervaya} appeared as Theorem 1 in \cite{DScand}. 
However, none of the current conditions (i.1)--(iii.1) were discussed in that paper, nor was 
the notion of a (non)ideal ideal introduced. 

\par Now we turn to the case of an integral order of smoothness, an issue that was not touched upon 
in \cite{DScand} altogether. This time, the algebras to be dealt with are $\HN=\{f:\,f^{(N)}\in H^\infty\}$. 
The phenomenon we are interested in may only occur when $N\ge2$, so it will be convenient to write 
$N=n+1$ with $n\in\N$. Given $n$ and a Blaschke product $B$, we put 
$$I_k:=\{f\in\hnone:\,fB^k\in\hnone\},\qquad k\in\N.$$ 
Our intention is to study the $I_k$'s, as ideals/subspaces of $\hnone$, in the same spirit as 
the $\J_k$'s above. As before, we have 
\begin{equation}\label{eqn:chainint}
I_1\supset\dots\supset I_n\supset I_{n+1} 
\end{equation}
and 
\begin{equation}\label{eqn:tailint}
I_{n+1}=I_{n+2}=\dots. 
\end{equation}
Here, the inclusions \eqref{eqn:chainint} are due to the fact that $\hnone$ has the $f$-property 
(i.\,e., division by inner factors preserves membership in $\hnone$), as shown by Shirokov in 
\cite{Shi1}. The equalities \eqref{eqn:tailint} can likewise be deduced from Shirokov's results; see 
\cite{Shi1} or Lemma \ref{lem:shimult} below. 
\par The $\HN$ counterpart of Theorem \ref{thm:pervaya} reads as follows. 

\begin{thm}\label{thm:tseloe} Let $n\in\N$ and suppose $B$ is an $H^\infty_1$-interpolating Blaschke 
product with zeros $\{z_j\}$. The following are equivalent. 

\smallskip{\rm (i.2)} The ideals $I_1,\dots,I_n$ are all nonideal (as subspaces of $\hnone$). 

\smallskip{\rm (ii.2)} The inclusions $I_1\supset\dots\supset I_n\supset I_{n+1}$ are all proper. 

\smallskip{\rm (iii.2)} $I_n$ is nonideal. 

\smallskip{\rm (iv.2)} $I_{n+1}\ne I_n$. 

\smallskip{\rm (v.2)} The sequence $\{z_j\}$ satisfies \eqref{eqn:crucon}. 
\end{thm}

Of course, by saying that $B=B_{\{z_j\}}$ is an $H^\infty_1$-interpolating Blaschke product we mean that 
its zeros are simple and their closure $E:={\rm clos}\{z_j\}$ is an {\it $H^\infty_1$-interpolating set} in 
the natural sense. That is, every function $\ph:E\to\C$ satisfying \eqref{eqn:lipone} with $\be=1$ should be 
representable as $\ph=f|_E$ for some $f\in H^\infty_1$. 

\par While the proof of Theorem \ref{thm:pervaya} relies heavily on Theorem A, we could expect to prove 
Theorem \ref{thm:tseloe} by following the same pattern, should the appropriate factorization theorem (analogous 
to Theorem A) exist in the $\HN$ setting. Specifically, we would need to have some variant of condition 
(iv.A) at our disposal, playing a similar role. Unfortunately, no such thing seems to be readily available, and 
the next result is intended to fill that gap. In addition to $\HN$, also involved in the statement below is 
the space 
$$\bmoa_N:=\{f\in\bmoa:\,f^{(N)}\in\bmoa\},$$ 
which is mentioned for the sake of completeness. Here, as usual, $\bmoa$ stands 
for the analytic subspace of $\bmo=\bmo(\T)$, the class of functions of bounded mean 
oscillation on $\T$; see \cite[Chapter VI]{G}. 

\begin{thm}\label{thm:factint} Let $m$ and $N$ be positive integers with $m\ge N$. Given $f\in\HN$ 
and an inner function $\th$, the following are equivalent. 

\smallskip{\rm (i.3)} $f\th^m\in\HN$. 

\smallskip{\rm (ii.3)} $f\th^k\in\HN$ for all $k\in\N$. 

\smallskip{\rm (iii.3)} $f\th^k\in\bmoa_N$ for all $k\in\N$. 

\smallskip{\rm (iv.3)} For some $\eps\in(0,1)$, one has 
$$|f(z)|=O\left((1-|z|)^N\right),\qquad z\in\Omte.$$ 
\end{thm}

This last result is somewhat more delicate than Theorem A and calls for a new method of proof. Indeed, neither 
duality arguments (as in \cite{DSpb, DAmer}) nor the pseudoanalytic extension approach (as in \cite{Dyn2}) 
that worked for $\Al$ carry over to $\HN$. Our proof will be accomplished by combining some of 
Shirokov's techniques from \cite{Shi1} with those developed by the author. Finally, let us remark that the 
equivalence between (ii.3) and (iii.3) reflects an amusing \lq\lq self-improving property" (i.\,e., an 
automatic increase in smoothness), a phenomenon discussed in greater generality in \cite{Djfa}. 

\par Going back to the $\Al$ setting, we wish to discuss yet another aspect of the problem, 
namely, the construction of an \lq\lq ideal hull" for a nonideal ideal. Suppose, under 
the hypotheses of Theorem \ref{thm:pervaya}, that \eqref{eqn:crucon} is fulfilled; 
assume also that $1\le k\le n$. We know that the corresponding ideals $\J_k$ are then 
properly contained in each other and nonideal. Thus, for $f\in\Al$, no kind of \lq\lq ideal" smallness 
condition on $|f|$ -- in particular, no reasonable size condition on $|f(z_j)|$ -- can possibly 
be necessary and sufficient in order that $f\in\J_k$. At the same time, one feels that 
the values $|f(z_j)|$ {\it must} become appropriately small, as $j\to\infty$, whenever $f\in\J_k$. 
This motivates our search for a {\it necessary} condition, ideal in nature and involving the decrease 
rate of $|f(z_j)|$, that should hold for each $f\in\J_k$. Moreover, we want the condition to be 
sensitive enough to distinguish between $\J_k$ and $\J_{k-1}$. 

\par We have been able to find such a necessary condition in the case where $n/2<k\le n$. This is 
provided by part (a) of the theorem below, while part (b) shows that the condition is optimal. The latter 
deals with the full range $1\le k\le n$, and we strongly believe that the former should also extend 
to all of these $k$'s. 

\begin{thm}\label{thm:vtoraya} Let $n\in\N$, $n<\al<n+1$, and suppose $B$ is an 
$\Al$-interpolating Blaschke product with zeros $\{z_j\}$. 

\smallskip{\rm (a)} If $k$ is an integer with $n/2<k\le n$, then every $f\in\J_k$ satisfies 
\begin{equation}\label{eqn:decr}
\left|f(z_j)\right|\le\const\cdot d_j^{\al-k}\left(1-|z_j|\right)^k,\qquad j\in\N. 
\end{equation}

\smallskip{\rm (b)} For each $k=1,\dots,n$, there is a function $f\in\J_k$ with 
\begin{equation}\label{eqn:comp}
\left|f(z_j)\right|\asymp d_j^{\al-k}\left(1-|z_j|\right)^k,\qquad j\in\N. 
\end{equation}
\end{thm}

\par Here and throughout, the notation $U\asymp V$ means that the ratio $U/V$ lies between 
two positive constants. Those (hidden) constants in \eqref{eqn:comp} are of course independent 
of $j$, as is the constant in \eqref{eqn:decr}. 

\par If $n=1$, then the only possible value of $k$ is $1$ and the inequality $n/2<k$ is automatic. 
Now if $1\le n/2<k\le n$ and if \eqref{eqn:crucon} holds, then Theorem \ref{thm:vtoraya} tells us 
that the ideal ideal 
$$\widetilde\J_k:=\left\{f\in\Al:\,\left|f(z_j)\right|=
O\left(d_j^{\al-k}\left(1-|z_j|\right)^k\right),\,\,j\in\N\right\}$$ 
contains $\J_k$ but not $\J_{k-1}$. Indeed, \eqref{eqn:crucon} says that the quantity 
$d_j^{\al-k}\left(1-|z_j|\right)^k$ is essentially smaller than 
\begin{equation}\label{eqn:kminus1}
d_j^{\al-k+1}\left(1-|z_j|\right)^{k-1}, 
\end{equation}
whereas statement (b) of the theorem produces a function $f\in\J_{k-1}$ for which $|f(z_j)|$ has 
the same order of magnitude as \eqref{eqn:kminus1}. We have thus constructed an {\it ideal envelope}, 
namely $\widetilde\J_k$, of the nonideal ideal $\J_k$ without making it \lq\lq too much fatter". 
In fact, no smaller ideal ideal resulting from a stronger decrease condition on $|f(z_j)|$ would do. 

\par In conclusion, we point out a corollary of Theorem \ref{thm:vtoraya} that establishes a connection 
between $\J_k$, with $k$ as above, and the ideal 
$$\J_{-1}:=\{f\in\Al:f/B\in\La^\al\}.$$ 
This last result will also rely on the following characterization of $\J_{-1}$ that appears 
in \cite[Corollary 4.3]{DSpb}: for a function $f\in\Al$ ($0<\al<\infty$) and an interpolating 
Blaschke product $B=B_{\{z_j\}}$, 
\begin{equation}\label{eqn:iffo}
f\in\J_{-1}\iff\left|f(z_j)\right|=O\left((1-|z_j|)^\al\right). 
\end{equation}
(Here and below, \lq\lq interpolating" stands for \lq\lq$H^\infty$-interpolating", meaning that $\{z_j\}$ 
is an interpolating sequence for $H^\infty$; cf. \cite[Chapter VII]{G}. It is known that every 
$\Al$-interpolating Blaschke product is $H^\infty$-interpolating.) As a consequence of \eqref{eqn:iffo}, 
we see that $\J_{-1}$ is, under the current conditions, an ideal ideal of $\Al$. 

\begin{cor}\label{cor:minusone} Let $n\in\N$, $n<\al<n+1$, and suppose $B$ is an $\Al$-interpolating Blaschke 
product with zeros $\{z_j\}$. Given an integer $k$ with $n/2<k\le n$, one has $\J_k\sb\J_{-1}$ if and only if 
\begin{equation}\label{eqn:supisfin}
\sup_j\f{d_j}{1-|z_j|}<\infty. 
\end{equation}
\end{cor}

To prove the \lq\lq if" part, we combine \eqref{eqn:supisfin} with \eqref{eqn:decr} to get 
$$\left|f(z_j)\right|\le\const\cdot\left(1-|z_j|\right)^\al,\qquad j\in\N,$$ 
for every $f\in\J_k$. Then we invoke \eqref{eqn:iffo} to conclude that $\J_k\sb\J_{-1}$. 
Conversely, assuming \eqref{eqn:crucon} and taking an $f\in\J_k$ with property \eqref{eqn:comp}, 
we obtain 
$$\sup_j\f{|f(z_j)|}{(1-|z_j|)^\al}=\infty.$$ 
This specific $f$ is therefore not in $\J_{-1}$, by \eqref{eqn:iffo} again, and the \lq\lq only if" part 
follows. 

\par The remaining part of the paper contains some preliminary material on interpolating sets for 
$\Al$ and $\hnone$, a few lemmas, and finally the proofs of our main results. 

\section{Preliminaries on free interpolation in $\Al$ and $\HN$} 

Let $n$ be a nonnegative integer. Given $g\in\Al$, with $n<\al<n+1$, one has 
\begin{equation}\label{eqn:taylor}
\left|g^{(s)}(z)-\sum_{m=0}^{n-s}\f{g^{(s+m)}(w)}{m!}(z-w)^m\right|\le C|z-w|^{\al-s}
\quad(z,w\in\D,\,\,s=0,\dots,n),
\end{equation}
where $C=C_g$ is a constant independent of $z$ and $w$. If $g\in\hnone$, then \eqref{eqn:taylor} 
holds with $\al=n+1$. 

\par A closed set $E\sb{\rm clos}\,\D$ is said to be an {\it $\Al$-interpolating set} if every 
interpolation problem 
\begin{equation}\label{eqn:intprob}
g|_E=\ph_0,\quad g'|_E=\ph_1,\quad\dots,\quad g^{(n)}|_E=\ph_n
\end{equation}
has a solution $g\in\Al$, provided that the {\it data} $\ph_s:E\to\C$ satisfy 
\begin{equation}\label{eqn:taylorphi}
\left|\ph_s(z)-\sum_{m=0}^{n-s}\f{\ph_{s+m}(w)}{m!}(z-w)^m\right|\le C|z-w|^{\al-s}
\quad(z,w\in E,\,\,s=0,\dots,n)
\end{equation}
with some fixed $C>0$. 

\par Similarly, we call $E$ an {\it $\hnone$-interpolating set} if every interpolation problem 
\eqref{eqn:intprob} has a solution $g\in\hnone$ whenever the {\it data} $\ph_s:E\to\C$ ($0\le s\le n$) 
satisfy \eqref{eqn:taylorphi} for $\al=n+1$ and for some constant $C>0$. 

\par Of course, \eqref{eqn:taylorphi} just means that the $\ph_s$ obey the necessary conditions 
coming from \eqref{eqn:taylor}. The validity of \eqref{eqn:taylorphi}, with $n<\al\le n+1$, will be also 
expressed by saying that the $(n+1)$-tuple $(\ph_0,\dots,\ph_n)$ is an {\it $\al$-admissible jet} on $E$. 

\par The following characterization of $\Al$-interpolating sets was given by Dyn'kin in \cite{Dyn1}; 
see also \cite[Sect.\,3]{Dyn2}. We shall use the notation $\rho(\cdot,\cdot)$ for the pseudohyperbolic 
distance on $\D$, so that $\rho(z,w):=|z-w|/|1-\bar zw|$. 

\begin{thb} Let $\al>0$, $\al\notin\N$, and let $E$ be a closed subset of ${\rm clos}\,\D$. Then $E$ is 
an $\Al$-interpolating set if and only if it has the two properties below: 
\par{\rm (1.B)} The set $E\cap\D$ is separated, in the sense that 
$$\inf\{\rho(z,w):\,z,w\in E\cap\D,\,z\ne w\}>0.$$
\par{\rm (2.B)} There is a constant $c>0$ such that every arc $I\sb\T$ satisfies 
$$\sup_{\ze\in I}{\rm dist}\,(\ze,E)\ge c|I|,$$ 
where ${\rm dist}\,(\ze,E):=\inf_{z\in E}|\ze-z|$ and $|I|$ is the length of $I$. 
\end{thb}

Subsequently, Shirokov \cite{ShiMS} extended Theorem B to a larger scale of Lipschitz-type spaces involving 
general moduli of continuity. As a special case (namely, for the modulus of continuity $\om(t)=t$), his results 
provide a description of $\hnone$-interpolating sets, which can be stated as follows. 

\begin{thc} Suppose $n$ is a nonnegative integer and $E$ is a closed subset of ${\rm clos}\,\D$. Then $E$ is 
an $\hnone$-interpolating set if and only if it satisfies ${\rm (1.B)}_2$ and {\rm (2.B)}, where ${\rm (1.B)}_2$ 
is the condition that $E\cap\D$ be a union of two separated sets. 
\end{thc}

A detailed discussion of the geometric condition (1.B)$\&$(2.B) can be found in \cite[Sect.\,5]{Dyn1}. 
In particular, it is shown there that if $E$ is an $\Al$-interpolating set, then $E\cap\D$ is an 
interpolating set for $H^\infty$, i.\,e., 
\begin{equation}\label{eqn:carl}
\inf_j\prod_{l:\,l\ne j}\rho(z_j,z_l)>0
\end{equation}
for any enumeration $\{z_j\}$ of $E\cap\D$. Thus, every $\Al$-interpolating Blaschke product is also 
$H^\infty$-interpolating (or just \lq interpolating', in standard terminology), as we mentioned before. 
Consequently, an $\hnone$-interpolating (or equivalently, $H^\infty_1$-interpolating) Blaschke product 
$B$ can always be written as $B=B_1B_2$, where the two factors are interpolating Blaschke products. 
\par Finally, let us remark that each $\Al$- or $\hnone$-interpolating set $E$ satisfies the 
{\it Beurling--Carleson condition} 
$$\int_\T\log{\rm dist}\,(\ze,E)\,|d\ze|>-\infty$$ 
and is, therefore, a non-uniqueness set for $\Al$ (see \cite{C}). 

\section{Some lemmas}

Given a sequence $\{z_j\}\sb\D$, recall the notation 
$$d_j:=\inf\{|z_j-z_l|:l\in\N\setminus\{j\}\}.$$ 

\begin{lem}\label{lem:admis} Let $0\le k\le n<\al\le n+1$, where $k$ and $n$ are integers. 
Assume also that $\{z_j\}\sb\D$ is a sequence satisfying 
$z_j\ne z_l$ for $j\ne l$ and having no accumulation points in $\D$. Finally, 
write $E:={\rm clos}\{z_j\}$ and define, for $s=0,\dots,n$, 
the functions $\ph_s$ on $E$ by putting 
$$\ph_s\equiv0\quad\text{for}\quad s\ne k,\quad
\ph_k(z_j)=d_j^{\al-k}\quad(j=1,2,\dots)\quad\text{and}\quad\ph_k|_{E\cap\T}=0.$$ 
Then $(\ph_0,\dots,\ph_n)$ is an $\al$-admissible jet on $E$. 
\end{lem} 

\begin{proof} The functions $\ph_s$ being continuous on $E$, it suffices to check that 
\begin{equation}\label{eqn:tayphijl}
\left|\ph_s(z_l)-\sum_{m=0}^{n-s}\f{\ph_{s+m}(z_j)}{m!}(z_l-z_j)^m\right|\le C|z_l-z_j|^{\al-s}
\qquad (j,l\in\N,\,\,j\ne l)
\end{equation}
for $s=0,\dots,n$ and for some fixed $C>0$. (This is precisely \eqref{eqn:taylorphi} with $z=z_l$ 
and $w=z_j$.) We let LHS stand for the left-hand side of \eqref{eqn:tayphijl}, and we now estimate 
it by considering three cases as follows. 
\par If $0\le s<k$, then 
$$\text{\rm LHS}=\left|\f{\ph_k(z_j)}{(k-s)!}(z_l-z_j)^{k-s}\right|=\f1{(k-s)!}d_j^{\al-k}|z_l-z_j|^{k-s}
\le|z_l-z_j|^{\al-s},$$
where the final inequality is due to the obvious facts that $(k-s)!\ge1$ and $d_j\le|z_l-z_j|$. 
\par If $s=k$, then 
$$\text{\rm LHS}=\left|\ph_k(z_l)-\ph_k(z_j)\right|=\left|d_l^{\al-k}-d_j^{\al-k}\right|
\le|z_l-z_j|^{\al-k}=|z_l-z_j|^{\al-s},$$
since one clearly has $\max(d_l,d_j)\le|z_l-z_j|$. 
\par Finally, if $k<s\le n$, then $\text{\rm LHS}=0$. Thus, in all cases \eqref{eqn:tayphijl} holds 
with $C=1$. 
\end{proof}

\begin{lem}\label{lem:fbm} Suppose $f$ is an analytic function on $\D$, $B$ is a Blaschke 
product with zeros $\{z_j\}$, and $m$ is a nonnegative integer. Then 
$$(fB^m)^{(m)}(z_j)=f(z_j)\cdot(B^m)^{(m)}(z_j),\qquad j\in\N.$$ 
\end{lem}

\begin{proof} Indeed, 
\begin{equation*}
\begin{aligned}
(fB^m)^{(m)}(z_j)
&=\sum_{l=0}^m{m\choose l}f^{(m-l)}(z_j)(B^m)^{(l)}(z_j)\\
&=f(z_j)\cdot(B^m)^{(m)}(z_j),
\end{aligned}
\end{equation*} 
since $(B^m)^{(l)}(z_j)=0$ for $0\le l\le m-1$. 
\end{proof}

\begin{lem}\label{lem:blann} Suppose $B$ is an interpolating Blaschke product with zeros 
$\{z_j\}$, and let $\de=\de(B)$ be the value of the infimum in \eqref{eqn:carl}. Then, 
for $m\in\N$, 
$$|(B^m)^{(m)}(z_j)|\asymp(1-|z_j|)^{-m},\qquad j\in\N,$$ 
where the constants involved depend only on $m$ and $\de$. 
\end{lem}

\begin{proof} The inequality 
$$|(B^m)^{(m)}(z_j)|\le C_m(1-|z_j|)^{-m}$$
is clearly true because $B^m$ is an $H^\infty$-function of norm 1. 
\par To prove the reverse inequality 
\begin{equation}\label{eqn:lower}
|(B^m)^{(m)}(z_j)|\ge c_{m,\de}(1-|z_j|)^{-m}, 
\end{equation}
we proceed by induction. When $m=1$, \eqref{eqn:lower} is a well-known restatement of 
the fact that $\{z_j\}$ is an interpolating sequence. Next, assuming that \eqref{eqn:lower} is 
established for some value of $m$, we note that 
$$(B^{m+1})^{(m+1)}(z_j)=(m+1)\cdot(B^mB')^{(m)}(z_j)=(m+1)\cdot(B^m)^{(m)}(z_j)\cdot B'(z_j);$$ 
to check the last step, apply Lemma \ref{lem:fbm} with $f=B'$. Thus, 
$$\left|(B^{m+1})^{(m+1)}(z_j)\right|=(m+1)|(B^m)^{(m)}(z_j)||B'(z_j)|,$$ 
and the desired estimate 
$$|(B^{m+1})^{(m+1)}(z_j)|\ge\const\cdot(1-|z_j|)^{-m-1}$$ 
now follows from the induction hypothesis \eqref{eqn:lower}, combined with the $m=1$ case. 
\end{proof}

The next two lemmas are borrowed from \cite{DSpb}; see Lemma 2.2 and Theorem 2.4 of that paper. 

\begin{lem}\label{lem:dyader} Suppose that $f$ is analytic on $\D$, $\th$ is an inner 
function, $0<\eps<1$, $\al>0$, and $k\in\N$. If 
$$|f(z)|=O\left((1-|z|)^\al\right),\qquad z\in\Omte,$$ 
then 
$$|f^{(k)}(z)|=O\left((1-|z|)^{\al-k}\right),\qquad z\in\Om(\th,\,\eps/2).$$ 
\end{lem}

\begin{lem}\label{lem:dyashi} Suppose that $f$ is a non-null function in $\HN$, with $N\in\N$, and $\th$ is 
an inner function. The following are equivalent. 

\smallskip{\rm (i)} For some $\eps\in(0,1)$, one has 
$$|f(z)|=O\left((1-|z|)^N\right),\qquad z\in\Omte.$$ 

\smallskip{\rm (ii)} The set $\specth\cap\T$ has Lebesgue measure $0$, and 
$$|f(\ze)|=O\left((1/|\th'(\ze)|^N\right),\qquad\ze\in\T\setminus\specth.$$ 
\end{lem}

Finally, we list some of Shirokov's results from \cite{Shi1}. In particular, the next lemma comprises 
Theorems 1 and 3 of \cite{Shi1}, when specialized to the $\HN$ case. 

\begin{lem}\label{lem:shimult} Let $g\in\HN$, where $N\in\N$, and suppose $I$ is an inner function such 
that $g/I\in H^\infty$. Then $g/I\in\HN$. If, in addition, the zeros of $I$ in $\D$ -- if any -- are all 
of multiplicity at least $N$, then we also have $gI\in\HN$ (and hence $gI^k\in\HN$ for all $k\in\N$). 
\end{lem} 

The last statement in parentheses was not explicitly mentioned by Shirokov, but it follows readily from 
the preceding one by induction. We conclude by citing a restricted version of \cite[Lemma 4]{Shi1}. 

\begin{lem}\label{lem:shider} Given an inner function $\th$ and a point $\ze\in\T\setminus\specth$, write 
$$d_\th(\ze):=\dist(\ze,\,\specth)\qquad\text{and}\qquad
\tau_\th(\ze):=\min\left(d_\th(\ze),\,1/|\th'(\ze)|\right).$$ 
Then, for every $l\in\N$, there is a constant $c_l>0$ such that 
$$|\th^{(l)}(\ze)|\le c_l\cdot\{\tau_\th(\ze)\}^{-l},\qquad\ze\in\T\setminus\specth.$$
\end{lem}

\section{Proofs of Theorems \ref{thm:pervaya} and \ref{thm:tseloe}} 

\noindent{\it Proof of Theorem \ref{thm:pervaya}.} 
The implications (i.1)$\implies$(iii.1) and (ii.1)$\implies$(iv.1) are obvious. It is also clear 
that (iii.1) implies (iv.1), since $\J_{n+1}$ is an ideal subspace of $\Al$ (by Theorem A). 

\smallskip (iv.1)$\implies$(i.1). Let $f\in\J_n\setminus\J_{n+1}$, so that 
\begin{equation}\label{eqn:belnotbel}
fB^n\in\Al\quad\text{\rm but}\quad fB^{n+1}\notin\Al. 
\end{equation}
It follows that, for $k=1,\dots,n$, the function $g_k:=fB^{n-k}$ is in $\J_k$, while 
$Bg_k=fB^{n-k+1}$ is in $\Al\setminus\J_k$. Since $|Bg_k|\le|g_k|$ in $\D$, we conclude that 
$\J_k$ is nonideal. 

\smallskip (iv.1)$\implies$(ii.1). Once again, let $f\in\J_n\setminus\J_{n+1}$ and $g_k=fB^{n-k}$. 
For $k=1,\dots,n$, \eqref{eqn:belnotbel} tells us that $g_k\in\J_k\setminus\J_{k+1}$, so 
the inclusion $\J_{k+1}\sb\J_k$ is proper. 

\smallskip We now know that (i.1)$\iff$(ii.1)$\iff$(iii.1)$\iff$(iv.1), so it is the equivalence 
between (iv.1) and (v.1) that remains to be proved. 

\smallskip (iv.1)$\implies$(v.1). Suppose (v.1) fails, so that 
\begin{equation}\label{eqn:notfive}
\sup_j\f{d_j}{1-|z_j|}<\infty, 
\end{equation}
and let $f\in\J_n$. The function $g:=fB^n$ is then in $\Al$, and 
$$g(z_j)=g'(z_j)=\dots=g^{(n-1)}(z_j)=0$$ 
for all $j$. The inequality \eqref{eqn:taylor} with $s=0$ and $w=z_j$ therefore yields 
\begin{equation}\label{eqn:penis}
\left|g(z)-\f{g^{(n)}(z_j)}{n!}(z-z_j)^n\right|\le C|z-z_j|^\al,\qquad z\in\D.
\end{equation}
In particular, for any fixed $j$ and for $z=z_l$ ($l\ne j$), this gives 
$$\f{\left|g^{(n)}(z_j)\right|}{n!}\le C|z_l-z_j|^{\al-n}.$$ 
Taking the infimum over $l\in\N\setminus\{j\}$ and recalling \eqref{eqn:notfive}, we get 
\begin{equation}\label{eqn:vagina}
\f{\left|g^{(n)}(z_j)\right|}{n!}\le Cd_j^{\al-n}\le\widetilde C(1-|z_j|)^{\al-n}, 
\end{equation}
with a suitable constant $\widetilde C>0$. 
\par Now suppose $z\in\Om(B,\eps)$, where $\eps>0$ is appropriately small. Since $B$ is an 
interpolating Blaschke product, it follows (cf. \cite[Chapter X, Lemma 1.4]{G}) that there is 
a number $\la\in(0,1)$ and a zero $z_j$ of $B$ such that $\rho(z,z_j)<\la$. Here, both $\eps$ 
and $\la$ can be taken to depend only on the \lq Carleson constant' $\de=\de(B)$, defined as 
the infimum in \eqref{eqn:carl}. The inequality $\rho(z,z_j)<\la$ then implies that 
\begin{equation}\label{eqn:euc}
|z-z_j|\le c(1-|z|)\quad\text{\rm and}\quad(1-|z_j|)\le c(1-|z|)
\end{equation}
for some $c=c(\de)>0$. In view of \eqref{eqn:penis}, we have 
$$|g(z)|\le\f{|g^{(n)}(z_j)|}{n!}|z-z_j|^n+C|z-z_j|^\al,$$ 
and combining this with \eqref{eqn:vagina} and \eqref{eqn:euc}, we finally obtain 
$$|g(z)|\le\const\cdot(1-|z|)^\al,\qquad z\in\Om(B,\eps).$$ 
Theorem A now tells us that $gB^k\in\La^\al$ for all $k\in\Z$. In particular, $gB=fB^{n+1}\in\Al$ 
and so $f\in\J_{n+1}$. 
\par We have thus checked that \eqref{eqn:notfive} implies the inclusion $\J_n\sb\J_{n+1}$, and 
hence the equality $\J_n=\J_{n+1}$, contradicting (iv.1). 

\smallskip (v.1)$\implies$(iv.1). Since $\text{\rm clos\,}\{z_j\}$ is an $\Al$-interpolating set, 
we can apply Lemma \ref{lem:admis} with $k=n$ to find a function $g\in\Al$ satisfying 
\begin{equation}\label{eqn:intlastder}
g(z_j)=g'(z_j)=\dots=g^{(n-1)}(z_j)=0,\quad g^{(n)}(z_j)=d_j^{\al-n}
\end{equation}
for all $j$. This $g$ is then divisible by $B^n$, so that $g=fB^n$ with $f\in H^\infty$. In fact, 
we have $f\in\Al$ (because $\Al$ enjoys the $f$-property, see Section 1) and hence $f\in\J_n$. 
Next, we use Lemma \ref{lem:fbm} to rewrite the last equality from \eqref{eqn:intlastder} as 
$$f(z_j)\cdot(B^n)^{(n)}(z_j)=d_j^{\al-n}.$$ 
In view of (the trivial part of) Lemma \ref{lem:blann}, or just because $B^n\in H^\infty$, it follows that 
$$|f(z_j)|(1-|z_j|)^{-n}\ge\const\cdot d_j^{\al-n}.$$ 
Consequently, 
$$\sup_j\f{|f(z_j)|}{(1-|z_j|)^\al}\ge\const\cdot
\sup_j\left(\f{d_j}{1-|z_j|}\right)^{\al-n}=\infty,$$
where we have also used (v.1). Finally, we invoke Theorem A (specifically, the (i.A)$\implies$(iv.A) part 
with $\th=B$ and $m=n+1$) to conclude that $fB^{n+1}\notin\Al$. Thus $f\in\J_n\setminus\J_{n+1}$, 
which yields (iv.1) and completes the proof. \quad\qed

\medskip To prove Theorem \ref{thm:tseloe}, we may proceed in quite a similar fashion, as soon as we have 
Theorem \ref{thm:factint} at our disposal. Taking the latter result for granted and postponing its verification 
to the next section, we now describe the passage from the proof of Theorem \ref{thm:pervaya} above to that 
of Theorem \ref{thm:tseloe}. Basically, this reduces to the following adjustments. Throughout, change the tags 
(i.1), $\dots$, (v.1) to (i.2), $\dots$, (v.2), respectively; replace $\J_k$ by $I_k$, $\Al$ by $\hnone$, 
and $\al$ (except in $\Al$) by $n+1$; instead of Theorem A, refer to Theorem \ref{thm:factint} with $N=n+1$. 
Finally, a minor modification is needed to check that the set $\Om(B,\eps)$ is again contained in 
$\bigcup_j\{z:\rho(z,z_j)<\la\}$ for some $\eps$ and $\la$ in $(0,1)$. This time, one writes $B=B_1B_2$, 
where $B_1$ and $B_2$ are interpolating Blaschke products (see Section 2), and notes that 
$$\Om(B,\eps)\sb\Om(B_1,\sqrt\eps)\cup\Om(B_2,\sqrt\eps).$$

\section{Proof of Theorem \ref{thm:factint}} 

(i.3)$\implies$(ii.3). This implication is a consequence of Shirokov's results that are collected in Lemma 
\ref{lem:shimult}. Indeed, the last assertion of the lemma, when applied to $g:=f\th^m$ and $I:=\th^m$, 
tells us that $f\th^{km}\in\HN$ for all $k\in\N$. Together with the fact that $\HN$ enjoys the $f$-property 
(which is also contained in Lemma \ref{lem:shimult}), this yields (ii.3). 

\smallskip (ii.3)$\implies$(iii.3). This is obvious, since $\HN\sb\bmoa_N$. 

\smallskip (iii.3)$\implies$(iv.3). It clearly follows from (iii.3) that $f\th^{N+1}$ lies in $\bmoa_N$ and 
hence, {\it a fortiori}, in the (higher order) Zygmund class 
$$A^N=\left\{f\in H^\infty:\,|f^{(N+1)}(z)|=O\left((1-|z|)^{-1}\right),\,\,z\in\D\right\}.$$ 
To verify the inclusion $\bmoa_N\sb A^N$, as used here, recall that $\bmoa$ is contained in the Bloch space 
$\mathcal B$ (see \cite[Chapter VI]{G}). We have thus checked that $f\th^{N+1}\in A^N$; this done, we apply 
the (i.A)$\implies$(iv.A) part of Theorem A, with $\al=N$ and $m=N+1$, to arrive at (iv.3). 

\smallskip (iv.3)$\implies$(i.3). First let us prove that (iv.3) implies $f\th\in\HN$. We want to show that 
the function 
\begin{equation}\label{eqn:revleib}
(f\th)^{(N)}=\sum_{l=0}^N{N\choose l}f^{(N-l)}\th^{(l)}
\end{equation}
is bounded, and our plan is to check this for (the boundary values of) each summand involved. 
\par Using (iv.3) and Lemma \ref{lem:dyader}, we obtain 
\begin{equation}\label{eqn:derlevel}
|f^{(N-l)}(z)|=O\left((1-|z|)^{l}\right),\qquad z\in\Om(\th,\,\eps/2), 
\end{equation}
for $l=0,\dots,N$. Since $f^{(N-l)}\in H^\infty_l$, Lemma \ref{lem:dyashi} enables us to rewrite the 
preceding estimate as 
\begin{equation}\label{eqn:revest}
|f^{(N-l)}(\ze)|=O\left((1/|\th'(\ze)|^l\right),\qquad\ze\in\T\setminus\specth.
\end{equation}
We may of course assume that $f\not\equiv0$, so Lemma \ref{lem:dyashi} tells us also that $\T\setminus\specth$ 
is a set of full measure on $\T$. 
\par Now let $\ze\in\T\setminus\specth$, and let $\ze^*\in\specth$ be a point with 
$$|\ze-\ze^*|=\dist(\ze,\,\specth)=:d_\th(\ze).$$ 
We have then 
\begin{equation}\label{eqn:revtaylor}
\left|f^{(N-l)}(\ze)-\sum_{j=0}^{l-1}\f{f^{(N-l+j)}(\ze^*)}{j!}(\ze-\ze^*)^j\right|\le C|\ze-\ze^*|^l
=C\,d^l_\th(\ze),
\end{equation}
and hence 
\begin{equation}\label{eqn:revtay}
\left|f^{(N-l)}(\ze)\right|\le C\,d^l_\th(\ze)+\sum_{j=0}^{l-1}\f{|f^{(N-l+j)}(\ze^*)|}{j!}d^j_\th(\ze).
\end{equation}
(The admissible values of $l$ are again $0,1,\dots,N$. When $l=0$, it is understood that the sums in 
\eqref{eqn:revtaylor} and \eqref{eqn:revtay} equal $0$; the estimates then reduce to saying that 
$|f^{(N)}(\ze)|\le C$ for almost all $\ze\in\T$, in accordance with the hypothesis that $f\in\HN$.)  
\par Recalling that $\ze^*\in\specth\sb\text{\rm clos}\,\Om(\th,\,\eps/2)$, we now combine \eqref{eqn:derlevel} 
with the obvious inequality $1-|\ze^*|\le d_\th(\ze)$ to get 
\begin{equation}\label{eqn:nminusl}
\left|f^{(N-l+j)}(\ze^*)\right|\le\tilde C\left(1-|\ze^*|\right)^{l-j}\le\tilde C\cdot d_\th^{l-j}(\ze),
\qquad j=0,\,\dots,\,l-1,
\end{equation}
where $\tilde C$ is a suitable constant. Substituting the resulting estimate from \eqref{eqn:nminusl} 
into \eqref{eqn:revtay}, we see that 
$$\left|f^{(N-l)}(\ze)\right|\le\const\cdot d^l_\th(\ze).$$
Comparing this with \eqref{eqn:revest} gives 
\begin{equation}\label{eqn:crunl}
\left|f^{(N-l)}(\ze)\right|\le\const\cdot\tau^l_\th(\ze),\qquad\ze\in\T\setminus\specth,
\end{equation}
where 
$$\tau_\th(\ze):=\min\left(d_\th(\ze),\,1/|\th'(\ze)|\right).$$ 
In conjunction with Lemma \ref{lem:shider}, this shows that the products $f^{(N-l)}\th^{(l)}$ appearing on the 
right side of \eqref{eqn:revleib} are all essentially bounded on $\T$. Therefore, $f\th$ is indeed in $\HN$, 
as desired. 
\par Finally, proceeding by induction, we use the already established implication 
\begin{equation}\label{eqn:revimp}
\text{\rm(iv.3)}\implies f\th\in\HN 
\end{equation}
to deduce that (iv.3) actually implies (ii.3). To this end, we just apply \eqref{eqn:revimp} with $f$ successively 
replaced by $f\th$, $f\th^2$, etc. And since (ii.3) trivially implies (i.3), we are done. 

\section{Proof of Theorem \ref{thm:vtoraya}}

(a) Let $f\in\J_k$, where $n/2<k\le n$, and set $g:=fB^k$. We have then $g\in\Al$ and 
\begin{equation}\label{eqn:deepzero}
g(z_j)=g'(z_j)=\dots=g^{(k-1)}(z_j)=0,\qquad j\in\N.
\end{equation}
The inequalities \eqref{eqn:taylor}, applied with $z=z_l$ and $w=z_j$ ($j\ne l$), yield 
$$\left|g^{(s)}(z_l)-\sum_{m=s}^n\f{g^{(m)}(z_j)}{(m-s)!}(z_l-z_j)^{m-s}\right|\le C|z_l-z_j|^{\al-s}$$ 
for every integer $s$ in $[0,n]$. Now, for $s=0,\dots,n-k$, this further reduces to 
\begin{equation}\label{eqn:obama}
\left|\sum_{m=k}^n\f{g^{(m)}(z_j)}{(m-s)!}(z_l-z_j)^{m-s}\right|\le C|z_l-z_j|^{\al-s},
\end{equation}
in view of \eqref{eqn:deepzero}. (Note that the current values of $s$ do not exceed $k-1$, thanks to 
the assumption $k>n/2$.) Multiplying both sides of \eqref{eqn:obama} by $|z_l-z_j|^{s-k}$, we get 
\begin{equation}\label{eqn:babayaga}
\left|\sum_{m=k}^n\f{g^{(m)}(z_j)}{(m-s)!}(z_l-z_j)^{m-k}\right|\le C|z_l-z_j|^{\al-k}
\qquad(s=0,\dots,n-k). 
\end{equation}
Next, keeping $j$ and $l$ fixed (with $j\ne l$), we write 
\begin{equation}\label{eqn:system}
\sum_{m=k}^n\f{g^{(m)}(z_j)}{(m-s)!}(z_l-z_j)^{m-k}=:R_s\qquad(s=0,\dots,n-k). 
\end{equation}
We shall view \eqref{eqn:system} as a system of $n-k+1$ linear equations with the \lq unknowns' 
$$g^{(k)}(z_j),\quad g^{(k+1)}(z_j)\cdot(z_l-z_j),\quad\dots,\quad g^{(n)}(z_j)\cdot(z_l-z_j)^{n-k}$$ 
and with \lq constant terms' $R_0,\dots,R_{n-k}$. The coefficient matrix that arises, say $\M$, is then 
nonsingular. Indeed, 
$$\M=\M(k,n):=
\begin{pmatrix}
\f1{k!}&\f1{(k+1)!}&\dots&\f1{n!}\\
\f1{(k-1)!}&\f1{k!}&\dots&\f1{(n-1)!}\\
\dots&\dots&\dots&\dots\\ 
\f1{(2k-n)!}&\f1{(2k-n+1)!}&\dots&\f1{k!}
\end{pmatrix},$$ 
and one verifies that $\det\M\ne0$ (e.\,g., computing the determinant explicitly) by induction on $n-k$. 
By Cramer's rule, 
\begin{equation}\label{eqn:solution}
g^{(k)}(z_j)=\f{\det\M_1}{\det\M},
\end{equation}
where $\M_1$ is obtained from $\M$ by replacing its first column with $(R_0,\dots,R_{n-k})^T$. 
\par Since 
$$|R_s|\le\const\cdot|z_l-z_j|^{\al-k}\qquad(s=0,\dots,n-k),$$ 
as ensured by \eqref{eqn:babayaga}, while the entries of $\M$ depend only on $k$ and $n$, it follows 
that $\det\M_1$ admits a similar estimate. (To see why, expand the determinant along the first column.) 
Consequently, by \eqref{eqn:solution}, we also have 
$$\left|g^{(k)}(z_j)\right|\le\const\cdot|z_l-z_j|^{\al-k},$$ 
with a constant not depending on $j$ and $l$. Taking the infimum over $l\in\N\setminus\{j\}$ and 
noting that 
$$g^{(k)}(z_j)=f(z_j)\cdot(B^k)^{(k)}(z_j)$$
(by Lemma \ref{lem:fbm}), we deduce that 
$$|f(z_j)|\cdot|(B^k)^{(k)}(z_j)|\le\const\cdot d_j^{\al-k},\qquad j=1,2,\dots.$$ 
Finally, we combine this with the estimate 
$$|(B^k)^{(k)}(z_j)|\ge\const\cdot(1-|z_j|)^{-k}$$ 
(from Lemma \ref{lem:blann}) to arrive at \eqref{eqn:decr}. 

\smallskip (b) Since $\text{\rm clos\,}\{z_j\}$ is an $\Al$-interpolating set, 
Lemma \ref{lem:admis} enables us to solve the interpolation problem 
\begin{equation}\label{eqn:intkthder}
g^{(s)}(z_j)=0\quad(0\le s\le n,\,\,s\ne k),\quad g^{(k)}(z_j)=d_j^{\al-k},\qquad j\in\N, 
\end{equation}
with a function $g\in\Al$. This $g$ is therefore divisible by $B^k$, so that $g=fB^k$ with 
$f\in H^\infty$. We know that $f$ is then actually in $\Al$ and hence in $\J_k$. Furthermore, in 
view of Lemma \ref{lem:fbm}, the equality $g^{(k)}(z_j)=d_j^{\al-k}$ from \eqref{eqn:intkthder} 
takes the form 
$$f(z_j)\cdot(B^k)^{(k)}(z_j)=d_j^{\al-k}.$$ 
This, in conjunction with Lemma \ref{lem:blann}, implies that 
$$|f(z_j)|\asymp d_j^{\al-k}(1-|z_j|)^k,$$ 
as desired.

\end{document}